\documentclass[11pt,reqno]{amsart}
\usepackage{amsthm}
\usepackage{amssymb}
\usepackage{latexsym}
\usepackage{multicol}
\usepackage{verbatim,enumerate}
\usepackage[usenames]{color}
\usepackage[all, poly]{xy}
\usepackage[utf8]{inputenc}
\usepackage{hyperref}
\usepackage{amsmath, amscd}
\usepackage{soul}
\usepackage{tikz}
\usetikzlibrary{matrix,arrows,decorations.pathmorphing}
\usetikzlibrary{arrows}
\allowdisplaybreaks
\usepackage{dirtytalk,caption,subcaption}
\usetikzlibrary{positioning}
\usepackage{float}
\usepackage{wasysym}

\numberwithin{equation}{section} % optional

\newtheorem{thm}{Theorem}
\newtheorem{prop}{Proposition}

\newtheorem{cor}{Corollary}

\theoremstyle{definition}
\newtheorem{defn}{Definition}
\newtheorem{example}{Example}

\theoremstyle{remark}
\newtheorem{rem}{Remark}

\newcommand{\wt}{\operatorname{wt}}

\newcommand{\supp}{\operatorname{supp}}

\newcommand{\RAAG}{\operatorname{RAAG}}

\advance\textwidth by 1.2in  \advance\oddsidemargin by -.6in \advance\evensidemargin by -.6in

\newcounter{cnt}
 \makeatletter
\def\mydggeometry{\makeatletter\dg@YGRID=1\dg@XGRID=20\unitlength=0.003pt\makeatother}
\makeatother \theoremstyle{remark}

\numberwithin{equation}{section}
\makeatletter
\def\section{\def\@secnumfont{\mdseries}\@startsection{section}{1}%
  \z@{.7\linespacing\@plus\linespacing}{.5\linespacing}%
  {\normalfont\scshape\centering}}
\def\subsection{\def\@secnumfont{\bfseries}\@startsection{subsection}{2}%
  {\parindent}{.5\linespacing\@plus.7\linespacing}{-.5em}%
  {\normalfont\bfseries}}
\makeatother

\begin{document}

\title[]{Multivariate growth series of graph products of groups}

\author{Chaithra Pilakkat}
\address{Technical University of Munich, TUM School of Computation, Information and Technology, Department of Mathematics, Boltzmannstra\ss e~3, 85748 Garching bei M\"unchen, Germany.}
\email{chaithra.pilakkat@tum.de}
\thanks{}

\author{Venkatesh Rajendran}
\address{Department of Mathematics, Indian Institute of Science, Bangalore 560012, India}
\email{rvenkat@iisc.ac.in}
\thanks{RV's work is partially supported by CRG/2023/008780 of of ANRF and DST FIST
program-2021 [TPN-700661].
}

\begin{abstract}
Right-angled Artin groups (RAAGs) and right-angled Coxeter groups (RACGs) associated with finite simple graphs are fundamental objects in geometric group theory. Their one-variable growth series with respect to the standard generating sets was classically expressed by Chiswell in terms of the one-variable independence polynomial of the defining graph \cite{Chiswell94} with suitable substitutions of the variable.

In this paper, we investigate the multivariate growth series of graph products of groups and derive explicit formulas in terms of the multivariate independence polynomial of the underlying graph through suitable substitutions of variables. As special cases, we obtain multivariate growth series formulas for RAAGs and RACGs, thereby extending the classical one-variable identities.

We further show that the coefficients of the multivariate growth series of RAAGs and RACGs admit explicit descriptions in terms of the double-marked and marked chromatic polynomials of graphs. This connection reveals a rich interplay between growth series and graph coloring invariants. In particular, we obtain completely explicit formulas for all the coefficients in the case of chordal graphs, which 
include, for example, trees and complete graphs.\end{abstract}

\maketitle

\section{Introduction}

 Let \(G\) be a group generated by a finite symmetric generating set \(S\), that is,
\(S=S^{-1}\). For an element \(g\in G\), let
$\ell_S(g)$
denote the \emph{word length} of \(g\) with respect to \(S\), namely the minimum
number of generators from \(S\) required to express \(g\). The
\emph{growth series} of \(G\) with respect to \(S\) is the formal power series given by
\[
G(t)=\sum_{g\in G} t^{\ell_S(g)}.
\]
One of the central questions in algebraic combinatorics is to understand the growth series of groups, including determining when these series are rational and finding explicit formulas for them.

Let $\mathcal{G}$ be a finite simple graph. It is well known that the growth series of the right angled Artin group $RAAG({\mathcal{G}})$ and the right angled Coxeter group $RACG({\mathcal{G}})$ with respect to the standard sets of generators are very explicitly given in terms of one-variable independence polynomial of $\mathcal{G}$, see \cite[Prop. 17.4.2]{Davis} and \cite{Chiswell94, Amri-Athreya}. More precisely, we have 
$$RAAG(\mathcal{G})(t) = I_{\mathcal{G}}\left(\frac{-2t}{1+t}\right)^{-1} $$
where $I_{\mathcal{G}}(t)=\sum_{k\ge 0}a_kt^k$ is the one-variable independence polynomial of $\mathcal{G}$, and $a_k$ counts the number of independent or stable subsets of $\mathcal{G}$ of size $k$ for all $k\ge 0.$
Similarly, we have 
\[
RACG({\mathcal{G}})(t)
=I_{\mathcal{G}}\!\left(\dfrac{-t}{1+t}\right)^{-1}
\] for the right angled Coxeter group $RACG_{\mathcal{G}}$.

The main goal of this paper is to study the multivariate version of the growth series of right angled Artin groups and right angled Coxeter groups and write them in terms of multivariate independence  polynomial of
$\mathcal{G}$ with suitable substitutions of variables. For this purpose, we consider the graph products of groups which was first introduced by  Green \cite{Green} and consider their  multivariate  growth series. Graph products of groups constitute one of the most important classes of groups in geometric group theory. 
We note that the multivariate case for right-angled Coxeter groups was first accomplished in our previous  work \cite{CDV}. In that paper, our framework relied on the connection between right-angled Coxeter groups and the universal enveloping algebras of partially commutative Lie superalgebras, which allowed us to explicitly calculate the Hilbert series via the marked independence series of the underlying simple graph. In the present work, we expand this scope significantly by developing a purely group-theoretic approach that applies to all graph products of groups.  The one variable growth series of the graph products of groups with respect to the standard generating sets were classically expressed by Chiswell in terms of the one-variable independence polynomial of the defining graph, see \cite{Chiswell94}. We obtain multivariate generalizations of Chiswell formulas in this paper. 

To state our result, first we recall the definition of the graph products of groups from Green \cite{Green}. Given a simple graph $\mathcal{G}$ with a vertex set $I$ and an edge set $E$, along with a collection of vertex groups $\{G_i\}_{i \in I}$, the graph product of groups $\Gamma := \Gamma(\mathcal{G}, \{G_i\}_{i\in I})$ is defined as the quotient of the free product   $\ast_{i \in I} G_i$ by the normal subgroup that forces elements from distinct groups $G_i$ and $G_j$ to commute whenever $(i,j) \notin E$, that is $i$ and $j$ are not adjacent. In particular, graph products define a class of group products that naturally interpolate between free products and direct products. The graph products of groups naturally includes the right-angled Artin group $RAAG({\mathcal{G}})$ when every vertex group is infinite cyclic group ($G_i=\mathbb{Z}$, $i\in I$) and it includes the right-angled Coxeter group $RACG({\mathcal{G}})$ when every vertex group is of order two ($G_i=\mathbb{Z}/2\mathbb{Z}$, $i\in I$). 

Our main result (Theorem \ref{mainthm}) provides an explicit, unified formula for the multivariate growth series $H_{\Gamma}(\mathbf{x})$ of any graph product of groups, assuming only that the vertex groups possess well-defined weight functions. The definition of the multivariate growth series $H_{\Gamma}(\mathbf{x})$ can be found in Section \ref{growthdef}. We prove that $H_{\Gamma}(\mathbf{x})$ is completely determined by the multivariate independence series of $\mathcal{G}$ and the individual growth functions $H_i(x_i)$ of each vertex group ${G}_i$. More precisely, we have 
$$H_{\Gamma}(\mathbf{x}) = \left( \sum_{S \in \mathcal{I}(\mathcal{G})} \prod_{i \in S} \left( \frac{1}{H_i(x_i)} - 1 \right) \right)^{-1}
$$
where $\mathcal{I(G)}$ denotes the collection of all stable subsets of $\mathcal{G}$.

By applying suitable variable substitutions to this single formula, we obtain the multivariate formula for graph products of various interesting families of groups. Specifically, by substituting the growth function of the infinite cyclic group ($H_i(x_i) = \frac{1+x_i}{1-x_i}$), we obtain the multivariate growth series for right-angled Artin groups:
\[
H_{\text{RAAG}({\mathcal{G}}})(\mathbf{x}) = I_{\mathcal{G}}\left(\frac{-2\mathbf{x}}{1+\mathbf{x}}\right)^{-1}
\]
Similarly, by evaluating the formula with the growth function of the cyclic group of order two ($H_i(x_i) = 1+x_i$), we  recover the multivariate growth series for right-angled Coxeter groups:
\[
H_{\text{RACG}({\mathcal{G}})}(\mathbf{x}) = I_{\mathcal{G}}\left(\frac{-\mathbf{x}}{1+\mathbf{x}}\right)^{-1}
\]

The second part of the paper focuses on expressing the growth series of RAAGs and RACGs in terms of double-marked and marked chromatic polynomials of the underlying simple graphs. In particular, we write down explicit formulas for the growth series of RAAGs and RACGs when the underlying simple graphs are chordal. These formulas seem to be new to the literature as far as we know. Very explicitly, we have the following expression for the coefficients of the growth series when the underlying graph $\mathcal{G}$ is chordal:
\begin{align*}
H_{\RAAG(\mathcal G)}(\mathbf x)[\bold x^{\bold m}] &= \sum\limits_{\boldsymbol{\lambda}\in S(\bold m)}2^{L(\boldsymbol{\lambda})}(-1)^{|\mathbf{m}|+\sum_{j\in \mathrm{supp}(\bold{m})}\ell(\boldsymbol{\lambda}_{j})}\left(\prod _{j\in \mathrm{supp}(\bold{m})} \binom{\sum\limits_{i\in \mathcal{G}_j}\ell(\boldsymbol{\lambda}_{i})}{\ell(\boldsymbol{\lambda}_{j})} \frac{\ell(\boldsymbol{\lambda}_{j})!}{\prod_{k=1}^{\infty}(d^{\boldsymbol{\lambda}_{j}}_{k}!)}\right)\\
H_{RACG(\mathcal{G})}(\bold x) [\bold x^{\bold m}] &=  \sum\limits_{\boldsymbol{\lambda}\in S(\bold m)}(-1)^{|\mathbf{m}|+\sum_{j\in \mathrm{supp}(\bold{m})}\ell(\boldsymbol{\lambda}_{j})}\left(\prod _{j\in \mathrm{supp}(\bold{m})} \binom{\sum\limits_{i\in \mathcal{G}_j}\ell(\boldsymbol{\lambda}_{i})}{\ell(\boldsymbol{\lambda}_{j})} \frac{\ell(\boldsymbol{\lambda}_{j})!}{\prod_{k=1}^{\infty}(d^{\boldsymbol{\lambda}_{j}}_{k}!)}\right)
\end{align*}

here we fix the perfect elimination ordering on $I$ and  set $\mathcal{G}_r$ by the subgraph of $\mathcal{G}$ induced by the set of vertices 
 $\{i\in I\setminus \{r\}: \{i,r\}\in E, i<r\}\cup \{r\}$. For other notations used in these formulas, we
 refer the readers to look at the Section \ref{sect5.3}. The second formula on the growth series of the right angled Coxeter groups already appears in our previous work \cite{CDV}, we stated here for completeness. 
We remark that since our main theorem Theorem \ref{mainthm} places no structural restrictions on the vertex groups, it allows for the explicit computation of growth series for mixed graph products, such as those combining different types of cyclic groups or Coxeter groups. A detail study of these types of products of groups will be carried out in our future work. 

The paper is organized as follows. In Section \ref{sect2}, we fix notation and cover the necessary preliminaries on independence series, graph products, normal forms and growth series. In Section \ref{sect3}, we state and prove our main theorem regarding the multivariate growth series. In Section \ref{sect4}, we examine several consequences of our formula, explicitly deriving the growth series for right-angled Artin and Coxeter groups.
In Section \ref{sect5}, we define the double marked and marked chromatic polynomials and express the double marked and marked chromatic polynomials in terms of the ordinary chromatic polynomials of simple graphs. Using the results of Section \ref{sect5}, we give explicit formulas for the growth series of right angled Artin groups and right angled Coxeter groups when the underlying simple graphs are $PEO$-graphs (which includes the chordal graphs) in Section \ref{sect6}. At the end, in Section \ref{sect7}, we present our formulas explicitly for complete graphs, trees, paths, star graphs, and edgeless graphs.

\section{Preliminaries}\label{sect2}

We first fix some notation.
For a set \(I\), denote \(\mathbb{Z}_{\ge 0}^{\,I}\) by the set of all finitely supported tuples of non-negative integers indexed by \(I\). Denote
$\mathcal{P}_{\mathrm{mult}}(I)$ by the collection of all finite multi-subsets of \(I\). For a multi-subset $M\in \mathcal{P}_{\mathrm{mult}}(I)$, we denote $\sigma(M)$ by the underlying set of $M$. 
We work over a characteristic zero field $\mathbb F$ and denote the formal power series algebra generated by 
 the set of distinct commuting variables $\{x_i : i\in I\}$ by $\mathcal R =\mathbb F[[x_i: i\in I]]$.  For \(\mathbf{m}=(m_i : i\in I)\in \mathbb{Z}_{\ge 0}^{\,I}\) and \(f\in \mathcal R\), we write $$\text{
$
f[\mathbf{x}^{\mathbf{m}}]$
for the coefficient of
$\mathbf{x}^{\mathbf{m}}
=
\prod_{i\in I} x_i^{m_i}$ in \(f\),\ $ \text{and set} \ |\mathbf{m}| = \sum_{i \in I} m_i$}.$$
We also use $\supp(\mathbf{m}) = \{i \in I : m_i > 0\}$ for the support of given tuple $\bold m\in \mathbb Z_{\ge 0}^I$.
We fix a simple graph $\mathcal G=(I, E)$ with vertex set $I$ and edge set $E$ throughout the paper. \textit{We assume that $I$ is countable, and we identify $I$ with $\{1, \dots, n\}$ when $I$ is finite and having cardinality $n.$ }

\subsection{Graph Products of Groups}
Let $\{G_i : i\in I\}$ be a collection of groups; not necessarily non-isomorphic.  The graph product $\Gamma(\mathcal G, \{G_i\}_{i\in I}) $ is defined by 
\begin{equation*}
 \Gamma:=\Gamma(\mathcal G, \{G_i\}) = \left( \ast_{i \in I} G_i \right) \Big/ \langle [G_i, G_j]  \text{ whenever } (i,j) \notin E \rangle,
\end{equation*}
where $\left( \ast_{i \in I} G_i \right) $ is the free products of the groups $G_i, i\in I$, and  $\langle [G_i, G_j] \text{ whenever } (i,j) \notin E \rangle$ is the normal subgroup of 
$\left( \ast_{i \in I} G_i \right) $ generated by $[G_i, G_j] \text{ for } (i,j) \notin E$.

\textit{Our convention: Throughout this paper, we adopt the convention that, in the graph product of groups $\Gamma$, the elements of $G_i$ and $G_j$ commute whenever the vertices $i$ and $j$ are not adjacent in $\mathcal G$. This convention differs from that used in some parts of the literature, where the commuting relation is imposed whenever $i$ and $j$ are adjacent in $\mathcal{G}$. }

\noindent
\textit{
We adopt our convention because it naturally relates the growth series of graph products of groups to the independence series of $\mathcal{G}$, allowing us to express the growth series in terms of generalized chromatic polynomials of $\mathcal{G}$. The two conventions are equivalent up to replacing the underlying graph by its complement. Consequently, all results stated using one convention can be translated immediately into the other by passing to the complement graph, and no ambiguity should arise. }

\begin{example} We have the following two major examples:
\begin{enumerate}
    \item     \emph{the right angled Artin group} $RAAG(\mathcal G)$ associated to $\mathcal{G}$ is obtained from by taking $G_i = \mathbb Z$ for all $i \in I$, and
\item  \emph{the right angled  Coxeter group} $RACG(\mathcal G)$ associated to $\mathcal{G}$ is obtained from by taking $G_i = \mathbb Z/2\mathbb Z$ for all $i \in I$.
\end{enumerate}
\end{example}

\subsection{Normal form an element of graph products}
To establish a normal form for elements in a graph product—analogous to reduced 
words in free products—one uses the notion of a reduced word.

\begin{defn} 
    \begin{enumerate}
        \item Elements of each vertex group $G_i, i\in I$ are called \emph{syllables}.
        \item  A \emph{word} of an element $\gamma$ in $\Gamma$ is a finite sequence 
$\bold w = (g_1, g_2, \dots, g_k)$ of elements chosen from the vertex groups $g_i\in G_i, 1\le i\le k$, and  the sequence represents the product 
$g_1 g_2 \cdots g_k=\gamma \in \Gamma$. In this case, we say $\bold w$ represent $\gamma$ and we call $k$ as the length of $\bold w$.
%\item For $\gamma\in \Gamma$, a chosen word of gamma is denoted as $\bold w(\gamma)$. Note that we can have $\bold w\neq \bold w'$ such that $\bold w(\gamma)=\bold w'(\gamma)$.
\item Let $\bold w = (g_1, \dots,g_i, g_{i+1}, \dots, g_k)$ be a word. Suppose $g_i\in G_u$, $g_{i+1}\in G_v$ such that $\{u, v\}\notin E$, then word $\bold w' =  (g_1, \dots, g_{i+1}, g_i, \dots, g_k)$ is said be obtained from $\bold w$ by
 syllable shuffling. 
 \item More generally, $\bold w'$ is said to be obtained by $\bold w$ by  syllable shuffling if there is a sequence syllable shuffling operations as in (3) are applied on $\bold w$ to obtain $\bold w'.$
    \end{enumerate} 
\end{defn}

\subsection{}
Let $\gamma\in \Gamma$, we denote by $W(\gamma)$ %= \{\bold w =   (g_1, \dots, g_k):  g_1 g_2 \cdots g_k=\gamma\}$
the set of all words that represent $\gamma. $ We can perform the following elementary operations
on words representing $\gamma$:
\begin{enumerate}
    \item[(I)] Removing any syllable $g_i$ if $g_i = 1$, where $1$ is the identity element of the appropriate group. 
    \item[(II)] If two consecutive syllables $g_i$ 
    and $g_{i+1}$ belong to the same vertex group $G_v$, replace them with their 
    product $g_i g_{i+1} \in G_v$.
    \item[(III)]  If two consecutive syllables $g_i \in G_u$ 
    and $g_{i+1} \in G_v$ correspond to non-adjacent vertices ($\{u, v\} \notin E$), 
    interchange their order to $(g_{i+1}, g_i)$. This is same as syllable shuffling.
\end{enumerate}
Note that the resulting word obtained from applying any 
finitely many sequence of operations from $(I) - (III)$ (in any order) of a given word from $W(\gamma)$ must be again in $W(\gamma)$. We can start with an arbitrary word from $W(\gamma)$ and perform these operations and obtain 
a reduced word.
\begin{defn}
    A word $\bold w\in W(\gamma)$ is defined to be \emph{reduced} if its length cannot be shortened by applying any 
 operations from $(I) - (III)$.
Equivalently, $\bold w = (g_1, \dots, g_r)$ is reduced if 
\begin{itemize}
    \item $g_i\neq 1$ for all $1\le i\le k$
    \item whenever $i_s = i_t$ with $s < t$, there exists $u$, $s < u < t$, such that
\begin{equation*}
(i_s, i_u) \in E.
\end{equation*}
\end{itemize}
\end{defn}

The fundamental structural properties of these reduced words are governed 
by the following result, due to Green \cite{Green}.

\begin{prop}[The Normal Form]
Let $\Gamma$ be the graph products of a collection of vertex groups 
$\{G_i\}_{i\in I}$. Every non-trivial element $g$ in $\Gamma$ can be 
expressed as a product $$g = g_1 \cdots g_k,$$ where $\bold w = (g_1, \dots, g_k)$ 
is a reduced word. Furthermore, any other reduced word representing $g$ obtained from $\bold w$
by syllable shufflings.

\end{prop}

\subsection{Growth Series of graphs products of groups}\label{growthdef} 

Let $G$ be a group, not necessarily finite. A function $w: G\to \mathbb Z_{\ge 0}$ is said to be a weight function of $G$ if it satisfies the following conditions:
\begin{align*}
&(1) \ \ w(1)=0,\qquad
(2)\ \ w(g)>0 \text{ for } g\neq 1,\qquad
(3)\ \ w(g^{-1})=w(g)\ \text{for all } g\in G,\\
&(4)\ \ w(gh)\le w(g)+w(h)\ \text{for all } g,h\in G,\qquad
(5)\ \ \{g\in G\mid w(g)\le N\}\ \text{is finite for every } N.
\end{align*}
   \medskip
   \noindent
We fix a weight function $w_i$ on $G_i$ for each $i\in I. $
Define the growth series in one variable $x_i$ of $G_i$ with respect to the weight function $w_i$ by
\begin{equation*}
H_i(x_i) = \sum_{g \in G_i} x_i^{w_i(g)}.
\end{equation*}
Note that $H_i(x_i)\in \in \mathbb F[[x_i]]$ because of the condition $(5)$. Set $F_i(x_i) = H_i(x_i)-1$. 
 Let $\gamma\in \Gamma$ and $\bold w(\gamma)= (g_1, \dots, g_r)$
be a reduced word representing $\gamma$. 
Define
\begin{equation*}
x^\gamma := \prod_{j=1}^r x_{i_j}^{w_{i_j}(g_j)}
\end{equation*}
It is easy to see that $x^\gamma$
is well-defined as any two reduced words representing $\gamma$ differ by only sequence of  syllable shufflings.
Define the multivariate growth series of $\Gamma$ to be
\begin{equation*}
H_{\Gamma}(\mathbf{x}) = \sum_{\gamma \in \Gamma} x^\gamma \in \mathcal{R}.
\end{equation*}
Since $w(1)=0$ for the identity element $1\in \Gamma$, we have that $H_{\Gamma}(\mathbf{x})$ is invertible in $\mathcal{R}$, in particular we can talk about $H_{\Gamma}(\mathbf{x})^q$ for any integer $q\in \mathbb Z.$
Our main aim of this article to compute $H_{\Gamma}(\mathbf{x})$ in terms of the growth series of $G_i's.$

\section{Growth series formula}\label{sect3}
The following is the main theorem of this paper.
\begin{thm}\label{mainthm}
    Let $\mathcal{G}=(I,E)$ be a simple graph, and let $\mathcal{I}(\mathcal{G})$ denotes the set of all stable subsets of $\mathcal{G}$. The multivariate growth series $H_{\Gamma}(\mathbf{x})$ of the graph product $\Gamma =  \Gamma(\mathcal G, \{G_i\}_{i\in I}) $ is given by
\begin{equation}\label{equationthm}
H_{\Gamma}(\mathbf{x}) = \left( \sum_{S \in \mathcal{I}(\mathcal{G})} \prod_{i \in S} \left( \frac{1}{H_i(x_i)} - 1 \right) \right)^{-1}.
\end{equation}
%\textit{Equivalently,}
%\begin{equation*}
%H_{\Gamma}(\mathbf{x}) = \left( \sum_{S \in\mathcal{I}(\mathcal{G})} \prod_{i \in S} \frac{-F_i(x_i)}{1 + F_i(x_i)} \right)^{-1}.
%\end{equation*}
\end{thm}

\begin{proof}
    Set
\begin{equation*}
D(\mathbf{x}) = \sum_{S \in \mathcal I(\mathcal G)} \prod_{i \in S} \left( \frac{1}{H_i(x_i)} - 1 \right).
\end{equation*}
It suffices to prove
\begin{equation*}
H_{\Gamma}(\mathbf{x})D(\mathbf{x}) = 1.
\end{equation*}
%Expand
%\begin{equation*}
%\frac{1}{H_i(x_i)} - 1 = \frac{-F_i(x_i)}{1 + F_i(x_i)} = \sum_{r \ge 1} (-1)^r F_i(x_i)^r.
%\end{equation*}
Recall that we have $F_i(x_i) = H_i(x_i) - 1$ and $\frac{1}{H_i(x_i)} - 1 = \frac{-F_i(x_i)}{1 + F_i(x_i)} = \sum_{r \ge 1} (-1)^r F_i(x_i)^r$. 
A monomial of $F_i(x_i)^r$ (for $r\ge1$) corresponds to an ordered tuple
\begin{equation*}
u_i = (h_{i,1}, \dots, h_{i,r}), \quad h_{i,j} \in G_i \setminus \{1\}.
\end{equation*}
Thus a term of
\begin{equation*}
H_{\Gamma}(\mathbf{x})D(\mathbf{x})
\end{equation*}
is represented by a pair $(\gamma, U)$ where $\gamma \in \Gamma$ and $U$ assigns to every vertex of an independent set $S \in  \mathcal I(\mathcal G)$ a nonempty tuple $u_i$ of the above form, for each $i\in S$.
For any such pair $(\gamma, U)$, we can define its sign and weight:
\begin{itemize}
    \item \textbf{Sign:} $\mathrm{Sign}(\gamma,U)$= $(-1)^{\sum_{i \in S} |u_i|},$ where $|u_i|$ denote the length of the tuple $u_i$
    \item \textbf{Weight:} $\mathrm{wt}(\gamma,U) = x^\gamma \prod_{i \in S} \prod_{h \in u_i} x_i^{w_i(h)}$
\end{itemize}

Let $\mathrm{Init}(\gamma)$ be the set of vertices $i$ such that some reduced word of $\gamma$ begins with a syllable from $G_i$. We can then define the following sets:
\begin{equation*}
A(\gamma, U) = \{i \in \mathrm{Init}(\gamma) : \{i\} \cup S \in \mathcal{I}(G)\}
\end{equation*}
\begin{equation*}
L(\gamma, U) = A(\gamma, U) \cup S
\end{equation*}
where $S$ is the independent set corresponding to $U$.
If $L(\gamma, U) = \emptyset$, it follows that $\gamma = 1$ and $U = \emptyset$. \textit{
We repeatedly use the following key observation without mentioning further:
let \(i\in \operatorname{Init}(\gamma)\), then there is a unique syllable
\(g_i(\gamma)\in G_i\setminus\{1\}\) such that some reduced word for
\(\gamma\) begins with \(g_i(\gamma)\). Moreover, if
\(i,j\in\operatorname{Init}(\gamma)\) and \(i\ne j\), then the vertex groups
\(G_i\) and \(G_j\) commute.
}

Let $M_\Gamma$ denote the set of all $(\gamma,U)$ such that $L(\gamma,U)\neq \emptyset$.
To proceed, fix a total order on the vertex set $I$. For any nontrivial pair $(\gamma, U)$, we define the minimal element:
\begin{equation*}
i_0 = \min L(\gamma, U).
\end{equation*}

\item[Case 1: $i_0 \in A(\gamma,U)$.] 
    Since $i_0 \in \mathrm{Init}(\gamma)$, we can fix a reduced word $w$ for $\gamma$ of the form
    \begin{equation*}
        w = (g_1, g_2, \dots, g_n), \quad \text{where } g_1 \in G_{i_0} \setminus \{1\}.
    \end{equation*}

    We shift this first syllable out of the group element by setting $\gamma' = g_1^{-1}\gamma$ (so that $\gamma'$ has the reduced expression $(g_2, \dots, g_n)$), and update the tuple configuration $U$ to $U'$ according to the following rules:
    \begin{itemize}
        \item If $i_0 \notin S$, we add $i_0$ to the  independent set $S$ by setting $S' = S \cup \{i_0\}$ and initializing the singleton tuple $u'_{i_0} = (g_1)$. The fact that $S' \in  \mathcal I(\mathcal G)$ follows directly from the definition of $A(\gamma,U)$.
        \item If $i_0 \in S$ and $u_{i_0} = (h_1, h_2, \dots, h_r)$, we prepend $g_1$ to the existing tuple, replacing $u_{i_0}$ with $u'_{i_0} = (g_1, h_1, h_2, \dots, h_r)$.
    \end{itemize}

    \item[Case 2: $i_0 \notin A(\gamma,U)$.] 
    By the definition of $L(\gamma, U)$, this assumption implies that $i_0 \in S$. Let $u_{i_0} = (h_1, h_2, \dots, h_r)$. We shift the first element $h_1$ from the tuple to the front of the group element by setting 
    \begin{equation*}
        \gamma' = h_1\gamma.
    \end{equation*}
    The tuple configuration is then updated to $U'$ by altering the component at $i_0$:
    \begin{itemize}
        \item If $r > 1$, we replace $u_{i_0}$ with the truncated tuple $u'_{i_0} = (h_2, \dots, h_r)$.
        \item If $r = 1$, we remove the vertex $i_0$ from the active independent set, mapping $S' = S \setminus \{i_0\}$. Also this $S'\in \mathcal{I(G)}$.
    \end{itemize}
    To see that this step preserves the structural properties of our framework, let $w = (g_1, g_2, \dots, g_n)$ be our fixed reduced word for $\gamma$. We claim that the concatenated word $w' = (h_1, g_1, g_2, \dots, g_n)$ is a reduced word for $\gamma'$. Indeed, if $w'$ were not reduced, there would exist some index $k$ such that $g_k \in G_{i_0} \setminus \{1\}$ commutes with all $g_j$ for $1 \le j < k$. This would imply $i_0 \in \mathrm{Init}(\gamma)$. Because $i_0 \in S$ and $S \in  \mathcal I(\mathcal G)$, it follows that $\{i_0\} \cup S = S \in  \mathcal I(\mathcal G)$, meaning $i_0 \in A(\gamma,U)$, which  contradicts our hypothesis. Consequently, $w'$ is a valid reduced expression for $\gamma'$, and we obtain:
    \begin{equation*}
        x^{\gamma'} = x_{i_0}^{w_{i_0}(h_1)} x^\gamma.
    \end{equation*}

In both cases, let $U'$ denote the resulting tuple configuration, which is now indexed by the updated independent set $S' \in \mathcal{I}(\mathcal{G})$. We define the map $\Phi:M_\Gamma\to M_\Gamma$ by 
\begin{equation*}
    \Phi(\gamma,U) = (\gamma',U').
\end{equation*}
To see that $\Phi$ reverses the sign, observe that in each subcase, the total number of elements across all active tuples changes by exactly $\pm 1$. In Case 1, a syllable ($g_1$) is appended to the configuration $U'$, either by activating a new vertex index or by increasing the length of an existing tuple. Conversely, in Case 2, a syllable ($h_1$) is removed from the configuration $U$, either by tuple truncation or by the complete deactivation of the vertex index $i_0$. It follows that $\sum_{i \in S'} |u'_i| = \sum_{i \in S} |u_i| \pm 1$, yielding
\begin{equation*}
\operatorname{Sign}(\gamma',U') = (-1)^{\sum_{i \in S'} |u'_i|} = (-1)^{\left(\sum_{i \in S} |u_i|\right) \pm 1} = -\operatorname{Sign}(\gamma,U).
\end{equation*}

To demonstrate weight preservation, we monitor the simultaneous changes in the group element monomial and the tuple product. In Case 1, the removal of the initial syllable $g_1 \in G_{i_0}$ from $\gamma$ reduces the group element weight by a factor of $x_{i_0}^{w_{i_0}(g_1)}$, so that $x^{\gamma'} = x_{i_0}^{-w_{i_0}(g_1)} x^\gamma$. Concurrently, the addition of $g_1$ to the tuple configuration at $i_0$ scales the tuple product precisely by $x_{i_0}^{w_{i_0}(g_1)}$. Combining these modifications, we obtain
\begin{equation*}
\operatorname{wt}(\gamma',U') = \left(x_{i_0}^{-w_{i_0}(g_1)} x^\gamma\right) \cdot \left( x_{i_0}^{w_{i_0}(g_1)} \prod_{i \in S} \prod_{h \in u_i} x_i^{w_i(h)} \right) = \operatorname{wt}(\gamma,U).
\end{equation*}
In Case 2, shifting $h_1 \in G_{i_0}$ to the front of the group element scales its weight by a factor of $x_{i_0}^{w_{i_0}(h_1)}$, yielding $x^{\gamma'} = x_{i_0}^{w_{i_0}(h_1)} x^\gamma$. Concurrently, removing $h_1$ from the tuple configuration eliminates its corresponding weight from the product, which implies
\begin{equation*}
\operatorname{wt}(\gamma',U') = \left(x_{i_0}^{w_{i_0}(h_1)} x^\gamma\right) \cdot \left( x_{i_0}^{-w_{i_0}(h_1)} \prod_{i \in S} \prod_{h \in u_i} x_i^{w_i(h)} \right) = \operatorname{wt}(\gamma,U).
\end{equation*}

\smallskip
\noindent\textit{  Proof of $L(\gamma', U') = L(\gamma, U)$ under Case 1:}
Suppose $(\gamma, U)$ satisfies the conditions of Case 1, where $i_0 \in A(\gamma, U)$. We claim that $L(\gamma', U') = L(\gamma, U)$. If $i = i_0$, then $i_0 \in L(\gamma, U)$ by assumption, and because Case 1 maps $i_0$ directly into $S'$, we have $i_0 \in S' \subseteq L(\gamma', U')$. For any vertex $i \neq i_0$, we verify equality via double inclusion:
\begin{itemize}
    \item \textbf{Inclusion $L(\gamma, U) \subseteq L(\gamma', U')$:} Let $i \in L(\gamma, U) \setminus \{i_0\}$. If $i \in S$, then since $S \subseteq S'$, it follows immediately that $i \in S' \subseteq L(\gamma', U')$. If $i \in A(\gamma, U) \setminus \{i_0\}$, then $i \in \operatorname{Init}(\gamma)$ and $\{i\} \cup S \in \mathcal{I}(\mathcal{G})$. Thus we can find a reduced word of $\gamma'$ such that it begins with $g_k$ where $g_k\in G_i$, so $i \in \operatorname{Init}(\gamma')$. Because $i$ and $i_0$ are distinct elements of $\operatorname{Init}(\gamma)$, the vertices $i$ and $i_0$ must mutually commute, meaning $\{i, i_0\} \in \mathcal{I}(\mathcal{G})$.  Since $\{i\} \cup S \in \mathcal{I}(\mathcal{G})$ and $\{i, i_0\} \in \mathcal{I}(\mathcal{G})$, we have $\{i\} \cup S' = \{i\} \cup S \cup \{i_0\} \in \mathcal{I}(\mathcal{G})$. Thus, $i \in A(\gamma', U') \subseteq L(\gamma', U')$.

    \medskip
    \item \textbf{Inclusion $L(\gamma', U') \subseteq L(\gamma, U)$:} Let $i \in L(\gamma', U') \setminus \{i_0\}$. If $i \in S'$, then since $i \neq i_0$, it must be that $i \in S \subseteq L(\gamma, U)$. If $i \in A(\gamma', U') \setminus \{i_0\}$, then $i \in \operatorname{Init}(\gamma')$ and $\{i\} \cup S' \in \mathcal{I}(\mathcal{G})$. Since $i_0 \in S'$, the stability of the subset $\{i\} \cup S'$ implies that $i$ and $i_0$ commute. Hence, the syllable $g_k \in G_i$ can be moved past $g_1$ to the front of $\gamma = g_1 \gamma'$, proving $i \in \operatorname{Init}(\gamma)$. Finally, the containment $S \subseteq S'$ yields $\{i\} \cup S \in \mathcal{I}(\mathcal{G})$, which means $i \in A(\gamma, U) \subseteq L(\gamma, U)$.
\end{itemize}

\smallskip
\noindent\textit{ Proof of $L(\gamma', U') = L(\gamma, U)$  under Case 2:}
Suppose instead that $(\gamma, U)$ satisfies the conditions of Case 2, where $i_0 \in S \setminus A(\gamma, U)$. We claim that $L(\gamma', U') = L(\gamma, U)$. For the minimal index itself, $i_0 \in S \subseteq L(\gamma, U)$ by definition. Under the map, if $r > 1$, then $i_0 \in S' = S \subseteq L(\gamma', U')$. If $r = 1$, then $i_0 \notin S'$, but because the syllable $h_1 \in G_{i_0}$ is shifted to the front of the group element, $i_0 \in \operatorname{Init}(\gamma')$. Since $\{i_0\} \cup S' = S \in \mathcal{I}(\mathcal{G})$, we obtain $i_0 \in A(\gamma', U') \subseteq L(\gamma', U')$. For any vertex $i \neq i_0$, we verify the equality via double inclusion:
\begin{itemize}
    \item \textbf{Inclusion $L(\gamma, U) \subseteq L(\gamma', U')$:} Let $i \in L(\gamma, U) \setminus \{i_0\}$. If $i \in S$, then since $i \neq i_0$ and $S' \in \{S, S \setminus \{i_0\}\}$, it remains true that $i \in S' \subseteq L(\gamma', U')$. If $i \in A(\gamma, U) \setminus \{i_0\}$, then $i \in \operatorname{Init}(\gamma)$ and $\{i\} \cup S \in \mathcal{I}(\mathcal{G})$. Since $i_0 \in S$, the independence condition forces $i$ and $i_0$ to commute, implying the syllable of $i$ commutes with $h_1 \in G_{i_0}$. Thus, $i \in \operatorname{Init}(\gamma' = h_1 \gamma)$. Furthermore, because $S' \subseteq S$, the relation $\{i\} \cup S \in \mathcal{I}(\mathcal{G})$ guarantees $\{i\} \cup S' \in \mathcal{I}(\mathcal{G})$, which proves $i \in A(\gamma', U') \subseteq L(\gamma', U')$.

    \medskip
    \item \textbf{Inclusion $L(\gamma', U') \subseteq L(\gamma, U)$:} Let $i \in L(\gamma', U') \setminus \{i_0\}$. If $i \in S'$, then since $S' \subseteq S$, we immediately have $i \in S \subseteq L(\gamma, U)$. If $i \in A(\gamma', U') \setminus \{i_0\}$, then $i \in \operatorname{Init}(\gamma')$ and $\{i\} \cup S' \in \mathcal{I}(\mathcal{G})$. In Case 2, we know that $i_0$ always belongs to $ \operatorname{init}(\gamma')$. This forces $i$ to commute with $i_0$, therefore we have  $i \in \operatorname{Init}(\gamma = h_1^{-1} \gamma')$. To see that $\{i\} \cup S \in \mathcal{I}(\mathcal{G})$, note that if $r > 1$, then $S = S'$ and the condition is trivial. If $r = 1$, then $S = S' \cup \{i_0\}$; since $\{i\} \cup S' \in \mathcal{I}(\mathcal{G})$ and $i$ commutes with $i_0$, it follows that $\{i\} \cup S \in \mathcal{I}(\mathcal{G})$. Thus, $i \in A(\gamma, U) \subseteq L(\gamma, U)$.
\end{itemize}
Thus, the identity $L(\gamma', U') = L(\gamma, U)$ holds universally.

To complete the argument, it remains to verify that the map $\Phi$ is an involution on $M_\Gamma$. Suppose first that a pair $(\gamma,U) \in M_\Gamma$ satisfies the conditions of Case~1, where $i_0 = \min L(\gamma,U) \in A(\gamma,U)$. As established above, we have  $L(\gamma', U') = L(\gamma, U)$, and hence $\min L(\gamma', U') = i_0$. This immediately implies that $i_0 \notin A(\gamma',U')$. Consequently, the image $(\gamma', U')$ satisfies the precise defining hypotheses of Case~2 with respect to the same minimal active index $i_0$. Applying the rules of Case~2 to $(\gamma', U')$ shifts the first element of the tuple at $i_0$ (which is $g_1$) back to the front of the group element, setting $\gamma'' = g_1\gamma' = \gamma$, and perfectly reverses the structural modification of the tuple configuration. This yields $\Phi(\gamma', U') = (\gamma, U)$.

Conversely, suppose that $(\gamma,U)$ satisfies the conditions of Case~2, so that the minimal index satisfies $i_0 \in S \setminus A(\gamma,U)$. In the resulting pair $(\gamma',U')$, the minimal index is again preserved. Moreover, because the initial element $h_1 \in G_{i_0}$ of the tuple is shifted to the front of the group element, we have $i_0 \in \operatorname{Init}(\gamma')$. Since $\{i_0\} \cup S' = S \in \mathcal{I}(\mathcal{G})$, it follows that $i_0 \in A(\gamma',U')$. Therefore, the pair $(\gamma',U')$ satisfies the hypotheses of Case~1. Applying the rules of Case~1 to $(\gamma',U')$ removes this initial syllable $h_1$ from $\gamma'$ and prepends it back to the tuple configuration at $i_0$, which exactly recovers the original pair; that is, $\Phi(\gamma', U') = (\gamma, U)$.

Consequently, $\Phi$ defines a sign-reversing, weight-preserving involution on the set of all pairs $(\gamma, U) \neq (1, \emptyset)$. It follows that all nontrivial terms cancel out in pairs within the summation. The only surviving contribution arises from the pair $(1, \emptyset)$, whose weight is exactly $1$. This completes the proof.
\end{proof}

\section{Growth series of right angled Artin and Coxeter groups}\label{sect4}
In this section, we record several important consequences of Theorem \ref{mainthm}. First we need a few definitions from basic graph theory. Recall that $\mathcal{G}=(I, E)$ is a simple graph with countably many vertices $I.$
\begin{defn}
\begin{enumerate}
    \item A subset $S$ of $I$ is called stable or independent subset if the subgraph spanned by $S$ has no edges in it. 
    \item We denote $\mathcal I(\mathcal{G})$ by the set of all \textit{finite stable subsets} of $\mathcal{G}$, the empty set is independent by convention.
    \item We denote $\mathcal I_{mult}(\mathcal{G})$ by the set of all finite stable multi-subsets of $\mathcal{G}$, means it is the collection of multi-subsets of $I$ whose underlying sets are stable. We have the natural inclusion $\mathcal I(\mathcal{G})\hookrightarrow{} \mathcal I_{mult}(\mathcal{G}).$
\end{enumerate}
    
\end{defn}

\begin{defn}[Independence series]
  The \textit{independence series} of $\mathcal{G}$ is defined to be 
 $$I(\mathcal G, \bold x) = \sum\limits_{S\in \mathcal{I(G)}}\left(\prod\limits_{i\in S}x_i\right)\in \mathcal{R}.$$
 Note that $I(\mathcal G, \bold x)$ is a polynomial if $\mathcal{G}$ is a finite simple graph. 

 \end{defn}
The following definition of marked independence series of $\mathcal{G}$ can be found here \cite[Section 2.5]{CDV}.
\begin{defn}[Marked independence series]
 The \emph{marked independence series} of \(\mathcal G\) is
\[
I_{mark}(\mathcal G,\mathbf{x})
=
\sum_{U\in\mathcal{I}_{\mathrm{mult}}(\mathcal G)}
\mathbf{x}(U)
\] where $
\mathbf{x}(U)
=
\prod_{i\in I}x_i^{m_i},
$ and \(m_i\) is the multiplicity of the vertex \(i\) in \(U\).
\end{defn}
It is not hard to see that we have the following interpretation of marked independence series:
$$I_{mark}(\mathcal G,\mathbf{x}) = I\left(\mathcal{G}, \frac{\bold x}{1-\bold x}\right),$$
where the right hand side obtained from the independence series $I(\mathcal{G}, \bold x)$ by the following substitution of variables  $x_i\mapsto \frac{x_i}{1-x_i}$ for each $i\in I$.
Indeed the motivation for the definition of marked independence series comes from this substitution of variables. 
The substitution $x_i\mapsto \frac{2x_i}{1-x_i}$ suggests us to define the following double-marked independence series.
\begin{defn}[Double-marked independence series]
     The \emph{double-marked independence series} of \(\mathcal G\) is
\[
I_{mark}^2(\mathcal G,\mathbf{x})
=
\sum_{U\in\mathcal{I}_{\mathrm{mult}}(\mathcal G)} \left(2^{|\sigma(U)|}
\prod_{i\in I}x_i^{m_i}\right) = I\left(\mathcal{G}, \frac{2\bold x}{1-\bold x}\right)
\] where  as before \(m_i\) is the multiplicity of the vertex \(i\) in \(U\) and $|\sigma(U)|$ is the number of elements in the underlying set of $U.$
\end{defn}
Since every element $f\in \mathcal{R}$ with constant term $1$ is invertible in $\mathcal{R}$, its inverse $f^{-1}$  (obtained via the geometric series expansion) also belongs to $\mathcal{R}$. In particular, we have
$$I(\mathcal G,\mathbf{x})^{-1}, I_{mark}(\mathcal G,\mathbf{x})^{-1}, I_{mark}^2(\mathcal G,\mathbf{x})^{-1}\in \mathcal R.$$

\medskip
\begin{example}
Let us consider the following graph $\mathcal{G}$.
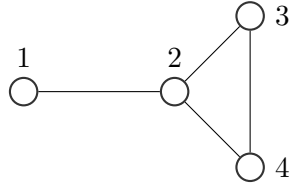
\begin{figure}[h]
\centering
\begin{subfigure}{0.49\textwidth}
\centering
    \begin{tikzpicture}
        \node (1) at (1,4) [circle, draw=black!80, thick, label=above:1] {};
    \node (2) at (3,4) [circle, draw=black!80, thick, label=above:2] {};
    \node (3) at (4,5) [circle, draw=black!80, thick, label=right:3] {};
    \node (4) at (4,3) [circle,draw=black!80, thick, label=right:4] {};

    \path[-] (1) edge node[left]{} (2);
    \path[-] (2) edge node[left]{} (3);
    \path[-] (2) edge node[left]{} (4);
    \path[-] (3) edge node[left]{} (4);
\end{tikzpicture}
\end{subfigure}
\end{figure}

Then we have 
\begin{align*}
    I(\mathcal{G}, \bold x)  = & \ 1+x_1+x_2+x_3+x_4+x_1x_3+x_1x_4,\\
    I_{mark}(\mathcal{G},\mathbf{x}) = & 1+\sum_{m=1}^{\infty}x_1^m+\sum_{m=1}^{\infty}x_2^m+\sum_{m=1}^{\infty}x_3^m+\sum_{m=1}^{\infty}x_4^m+\sum_{m, m'=1}^{\infty}x_1^mx_3^{m'}+\sum_{m, m'=1}^{\infty}x_1^mx_4^{m'},\\
    I^2_{mark}(\mathcal{G},\mathbf{x}) = & 1+\sum_{m=1}^{\infty}x_1^{2m}+\sum_{m=1}^{\infty}x_2^{2m}+\sum_{m=1}^{\infty}x_3^{2m}+\sum_{m=1}^{\infty}x_4^{2m}+\sum_{m, m'=1}^{\infty}2^{m+m'}x_1^{m}x_3^{m'}+\sum_{m, m'=1}^{\infty}2^{m+m'}x_1^{m}x_4^{m'}.
\end{align*}
\end{example}

\subsection{}  
Let \(G\) be a group generated by a finite set $S\subseteq G$ such that 
$S=S^{-1}$ and $1\notin S.$ The \emph{length function} with respect to \(S\) is the map
$\ell_S:G\longrightarrow \mathbb{Z}_{\ge 0}$
defined by
\[
\ell_S(g)
=
\min\left\{
k\ge 0
\;\middle|\;
g=s_1s_2\cdots s_k,\;
s_i\in S
\right\}.
\]
It is easy to see that length function $\ell_S$ is a weight function on $G$. We are interested in the corresponding spherical growth series of $G$, namely
$$H_G(x) =  \sum_{g\in G} x^{\ell_S(g)}. $$

When $I$ is finite and each $G_i$ is finitely generated, setting $x_i = t$ for all $i \in I$ in our multivariate growth series formula \eqref{equationthm} recovers the one variable growth series formula due to Chiswell \cite{Chiswell94} Corollary~\ref{chiswellcor}).

\begin{cor}\label{chiswellcor}

    Let $\mathcal{G}=(I,E)$ be a finite simple graph, and let $\mathcal{I}(\mathcal{G})$ denotes the set of all stable subsets of $\mathcal{G}$. The  one variable growth series $H_{\Gamma}(t)$ of the graph product $\Gamma =  \Gamma(\mathcal G, \{G_i\}_{i\in I}) $ is given by
\begin{equation}\label{equationcor}
H_{\Gamma}(t) = \left( \sum_{S \in \mathcal{I}(\mathcal{G})} \prod_{i \in S} \left( \frac{1}{H_i(t)} - 1 \right) \right)^{-1}.
\end{equation}
\end{cor}

The following examples are standard.
\begin{example}
    \begin{enumerate}
\item Let $G=\mathbb{Z}$ and $S=\{1,-1\}$. Then the associated length function is
$\ell_S(n)=|n|,\ \text{for all}\ n\in\mathbb{Z}.$ In this case, the growth series of $\mathbb Z$ is given by $$H_{\mathbb Z}(\bold x) = \sum_{n\in \mathbb Z}x^{|n|}=1+\frac{2x}{1-x}=\frac{1+x}{1-x}.$$

\item Let $G=\mathbb{Z}/n\mathbb{Z}$ and $
S=\{\,1+n\mathbb{Z},\, -1+n\mathbb{Z}\,\}.$
The associated weight function is given by
$\ell_S(k+n\mathbb Z)=\min\{k,n-k\},\ \text{for all}\ 0\le k<n.$ Thus, the growth polynomial is
\[
H_{\mathbb Z/n\mathbb Z}(x)=
\begin{cases}
\displaystyle
1+2\sum_{r=1}^{(n-1)/2}x^r,
& \text{if } n \text{ is odd},\\[2ex]
\displaystyle
1+2\sum_{r=1}^{n/2-1}x^r+x^{n/2},
& \text{if } n \text{ is even}.
\end{cases}
\]

\item Let $G=(W, S)$ be a Coxeter group with finite generating set $S.$ Then $\ell_S$ is the usual length function of the Coxeter group with respect to the generating set $S.$
\end{enumerate}
\end{example}

\subsection{} Recall that the right angled Artin group $RAAG(\mathcal{G})$ associated to $\mathcal{G}$ is defined by the group generated by $s_i, i\in I$ modulo the relations $s_is_j=s_js_i$ whenever $\{i, j\}\notin E.$
\begin{cor}
    Let $RAAG(\mathcal{G})$ be the right angled Artin group associated to $\mathcal{G}$. Then we have 
    $$H_{RAAG(\mathcal{G})}(\bold x) %= \left( \sum_{S \in \mathcal{I}(\mathcal{G})} \prod_{i \in S} \left( \frac{-2x_i}{1+x_i} \right) \right)^{-1}=I\left(\mathcal{G}, \frac{-2\bold x}{1+\bold x}\right)^{-1}
    =I_{mark}^2(\mathcal{G}, -\bold x)^{-1}.$$
\end{cor}
\begin{proof}
If we take $G_i=\mathbb Z$ for $i\in I$, then we have $\Gamma\cong RAAG(\mathcal{G})$. Since $H_i(x_i)=\frac{1+x_i}{1-x_i}$, we get $$\frac{1}{H_i(x_i)}-1=\frac{-2x_i}{1+x_i}.$$ Substituting this in the Equation \ref{equationthm}, we get the desired result. 
\end{proof}

\subsection{} Recall that the right angled Coxeter group $RACG(\mathcal{G})$ associated to $\mathcal{G}$ is defined by the group generated by $s_i, i\in I$ modulo the relations 
$$\text{
$s_i^2=1$ for all $i\in I$ and
$s_is_j=s_js_i$ whenever $\{i, j\}\notin E.$} $$
The following result was first proved in \cite[Theorem 4]{CDV}. %using the connection between the universal enveloping algebra of the partially commutative Lie superalgebra whose generater correspond to isotropic odd imaginary simple roots and the group algebra of the right angled Coxeter group. 
\begin{cor}
    Let $RACG(\mathcal{G})$ be the right angled Coxeter group associated to $\mathcal{G}$. Then we have 
    $$H_{RACG(\mathcal{G})}(\bold x) %= \left( \sum_{S \in \mathcal{I}(\mathcal{G})} \prod_{i \in S} \left( \frac{-x_i}{1+x_i} \right) \right)^{-1}=I\left(\mathcal{G}, \frac{-\bold x}{1+\bold x}\right)^{-1}
    =I_{mark}(\mathcal{G}, -\bold x)^{-1}.$$
\end{cor}
\begin{proof}
If we take $G_i=\mathbb Z/2\mathbb Z$ for $i\in I$, then we have $\Gamma\cong RACG(\mathcal G)$. Since $H_i(x_i)=1+x_i$, we get $$\frac{1}{H_i(x_i)}-1=\frac{-x_i}{1+x_i}.$$ Substituting this in the Equation \ref{equationthm}, we get the desired result. 
\end{proof} 

\begin{rem}
Note that Theorem~\ref{mainthm} is quite general and admits explicit specializations for several important classes of graph products of groups. In particular, analogous formulas can be derived when each $G_i$ is a cyclic group, or more generally, a finite Coxeter group. These cases will be investigated in a separate work.

In the present article, as indicated in the introduction, we restrict our attention to right-angled Artin groups and right-angled Coxeter groups. Our main objective is to express the coefficients of their growth series in terms of some generalizations chromatic polynomials of the underlying simple graph $\mathcal{G}$.
\end{rem}

%%%%%%%%%%%%%

\section{Marked and double-marked multi-colorings of a simple graph}\label{sect5}
As before, take $\mathcal{G}$ be a simple graph with a countable vertex set $I$ and edge set $E.$
We now recall the definition of marked multi-colorings of  $\mathcal{G}$ from \cite[Section 3.1]{CDV}. Throughout this subsection let $\bold{m}:=(m_i:i\in I)$ be a tuple of non-negative integers with finite support.
\begin{defn} Given $q\in\mathbb{N}$ we define the following.
\begin{enumerate}
        \item We call a map $_{\bold{m}}\Gamma^{\mathrm{mark}}_\mathcal{G}:I\rightarrow P^{\mathrm{mult}}(\{1,\dots,q\})$ a marked multi-coloring of $\mathcal{G}$ associated to $\bold{m}$ using at most $q$-colors if the following conditions are satisfied:\vspace{0,1cm}
        
        \begin{enumerate}[(i)]
            \item for all $i\in I$ we have $|_{\bold{m}}\Gamma^{\mathrm{mark}}_\mathcal{G}(i)|=m_i$,\vspace{0,1cm}
            \item for all $i,j\in I$ such that $(i,j)\in E$, we have $_{\bold{m}}\Gamma^{\mathrm{mark}}_\mathcal{G}(i)\cap {_{\bold{m}}\Gamma^{\mathrm{mark}}_\mathcal{G}(j)}=\emptyset$.
            
        \end{enumerate}
       % \item We say that two given $\prescript{}{1}{}\Bar{\tau}^{\bold{m}}_\mathcal{G}$ and $\prescript{}{2}{}\Bar{\tau}^{\bold{m}}_\mathcal{G}$ proper vertex super multi-coloring  of $\mathcal{G} $ associated to $\bold{m}$ using at most $q$-colors distinct if $\prescript{}{1}{}\Bar{\tau}^{\bold{m}}_\mathcal{G}(i)\neq \prescript{}{2}{}\Bar{\tau}^{\bold{m}}_\mathcal{G}(i)$ for some $i\in I$.
       % {\color{blue} Can we erase this part? I dont see the point.}
       \vspace{0,1cm}
       
        \item The number of marked multi-colorings of $\mathcal{G}$ associated to $\bold{m}$  using at most $q$-colors %is a polynomial in $q$, called the \textit{super-generalized chromatic polynomial} of $\mathcal{G}$ associated to $\bold{m}$, and it 
        is denoted by $_{\bold{m}}\Pi^{\mathrm{mark}}_\mathcal{G}(q)$.
    \end{enumerate}
\end{defn}
We recall from  \cite[Proposition 3.1]{CDV} that the function $_{\bold{m}}\Pi^{\mathrm{mark}}_\mathcal{G}(q)$ is indeed a polynomial in $q$, called  as the marked chromatic polynomial of $\mathcal{G}$ associated to $\mathbf m$. To state precisely, the expression of $_{\bold{m}}\Pi^{\mathrm{mark}}_\mathcal{G}(q)$ we need the following notion.

\begin{defn}
    We denote by $P_k(\bold{m},\mathcal{G})$ the set of all ordered $k$-tuples $(P_1,\dots,P_k)$ satisfying:
\begin{enumerate}
\item $P_r$ is a non-empty stable multi-subset of $I$ for all $1\leq k\leq r$ \medskip
    \item the disjoint union (as multi-sets) 
    $$P_1\sqcup \dots\sqcup P_k = \{i^{m_i} : i\in\mathrm{supp}(\bold{m})\},$$ i.e., $i$ appears exactly $m_i$ times for each $i\in \mathrm{supp}(\bold{m})$.
\end{enumerate} \end{defn} Since $\mathrm{supp}(\bold{m})$ is finite, $P_k(\bold{m},\mathcal{G})$ is also finite.
We have the following expression of the marked chromatic polynomial of $\mathcal{G}$ (see \cite[Proposition 3.1]{CDV}).
\begin{prop}\label{props6}
For each $q\in \mathbb{N}$ and $\mathbf{m}\in\mathbb{Z}_{\ge 0}^I$ with non-empty finite support, we have
    \begin{equation}
    \label{equation1}
    _{\bold{m}}\Pi^{\mathrm{mark}}_\mathcal{G}(q)=\sum_{k\geq 1}| P_k(\mathbf{m},\mathcal{G})| \binom{q}{k}.
\end{equation} \end{prop}
The following beautiful formula proved in \cite[Theorem 1]{CDV}.
\begin{thm}\label{expmarkedindp}
    Let $\mathcal{G}$ be a simple graph with countable vertex set $I$. For $q\in \mathbb{Z}$, we have as formal series
    $$I_{mark}(\mathcal{G},\mathbf{x})^q=\sum_{\mathbf{m}\in \mathbb{Z}_{\ge 0}^{I}} {_{\bold{m}}\Pi^{\mathrm{mark}}_\mathcal{G}(q)}\ \bold x^\mathbf{m}.$$
\end{thm} 
As an immediate corollary, the authors of \cite{CDV} obtained the following explicit formula for the grown series of $RACG(\mathcal{G})$  (see \cite[Theorem 4]{CDV}):
\begin{cor}\label{corracg}   Let $\mathcal{G}$ be a simple graph with countable vertex set $I$. For $q\in \mathbb{Z}$, we have as formal series
 $$H_{RACG(\mathcal{G})}(\bold x) = \sum_{\mathbf{m}\in \mathbb{Z}_{\ge 0}^{I}} (-1)^{|\bold m|} {_{\bold{m}}\Pi^{\mathrm{mark}}_\mathcal{G}(-1)}\ \bold x^\mathbf{m}.$$
\end{cor}

In particular, one is able to compute the growth series of $RACG(\mathcal{G})$ very explicitly for finite chordal graphs or more generally for $PEO$ graphs (see \cite[Corollary 6.3]{CDV}). %We will recall these formulas from \cite[Corollary 6.3]{CDV} for completeness. 
Our aim is to write down similar explicit formulas for the coefficients of $H_{RAAG(\mathcal{G})}(\bold x)$ for $PEO$-graphs. To achieve this, we need to know how to interpret 
 the coefficients of $H_{RAAG(\mathcal{G})}(\bold x)$ in terms of some generalized colorings of $\mathcal G.$ We will do this in upcoming sections.

\subsection{} In this section, we want to generalize the notion of marked colorings of $\mathcal{G}$ to double-marked colorings of $\mathcal{G}$ and obtain a formula  similar to the above one that is in Corollary \ref{corracg} for the right angled Artin groups. First we set the following notation: denote  $$\text{
$\mathcal{M}^2$ by the set of all pairs $(M, \epsilon)$ where  $M\in \mathcal{P}_{\text{mult}}(\{1, 2, \ldots, q\})$ and
$\epsilon: \sigma(M)\to \{+, -\}$.
}$$
\begin{defn}
Let $\mathcal{G} = (I, E)$ be a simple graph with countable vertex set $I$, and let $\mathbf{m} = (m_i:i\in  I)$ be a tuple of non-negative integers with finite support. For an integer $q \in \mathbb{N}$, a \emph{double-marked multi-coloring} of $\mathcal{G}$ associated with $\mathbf{m}$ using at most $q$ colors is an assignment 
$$_{\bold{m}}\Gamma^{2-\mathrm{mark}}_\mathcal{G}:I\rightarrow \mathcal{M}^2 $$
such that for each vertex $i \in I$,  we have $i\mapsto (M_i, \epsilon_i)$ where
\begin{enumerate}
    \item[(a)]      $
    M_i \in \mathcal{P}_{\text{mult}}(\{1, 2, \ldots, q\})$
    is a multi-set of colors, and 
    \item[(b)] $\epsilon_i: \sigma(M_i)\to  \{+, -\}$ (for each color $a \in \sigma(M_i)$ we assign a sign 
    $
    \epsilon_i(a) \in \{+, -\}$)
\end{enumerate}
subject to the following conditions:
\begin{enumerate}
    \item $|M_i| = m_i$ for all $i \in I$;
    \item if $\{i, j\} \in E$, then
    \[
    M_i \cap M_j = \emptyset.
    \]
\end{enumerate}
We denote the total number of distinct such double-marked multi-colorings of $\mathcal{G}$ associated with $\mathbf{m}$ using at most $q$ colors by 
$_{\mathbf{m}}\Pi_\mathcal{G}^{2}( q).$ We will soon prove that indeed this function is a polynomial in $q.$
\end{defn}

\begin{rem}
The sign is attached to each pair \((i,a)\in I\times [q]\) such that the color \(a\) occurs at the vertex \(i\).
Thus if \(a\) occurs with multiplicity \(r\ge1\) in \(M_i\), there are still only two choices for
the signs at \((i,a)\). This is precisely the convention that leads to the
substitution $
-\frac{2x_i}{1+x_i}.$
\end{rem}
\iffalse
We need the following definitions to proceed.
\begin{defn}
For \((P_1,\ldots,P_k)\in P_k(\mathbf m,\mathcal{G})\), define its doble-marked weight by
\[
\wt_2(P_1,\ldots,P_k)
=
\prod_{r=1}^k 2^{|\supp(P_r)|}.
\]
Define
\[
W_k^{(2)}(\mathbf m,\mathcal{G})
=
\sum_{(P_1,\ldots,P_k)\in P_k(\mathbf m,\mathcal{G})}
\wt_2(P_1,\ldots,P_k).
\]
That is,
\[
W_k^{(2)}(\mathbf m,G)
=
\sum_{(P_1,\ldots,P_k)\in P_k(\mathbf m,\mathcal{G})}
\prod_{r=1}^k 2^{|\supp(P_r)|}.
\]
\end{defn}\fi

\begin{defn}
Let $P$ be a multi-set of vertices such that its underlying set $\sigma(P)$ is an independent subset of $\mathcal{G}$. A \emph{marked independent multi-set} is a pair $\widetilde{P} = (P, \mu)$, where 
\[
\mu: \sigma(P) \longrightarrow \{+, -\}
\]
is a marking function assigning a  sign to each distinct vertex in the underlying set $\sigma(P)$.

Given a tuple of non-negative integers $\bold m$ with finite support, we denote by $\widetilde{P}_k(\bold m, \mathcal{G})$ the set of all ordered $k$-tuples $(\widetilde{P}_1, \dots, \widetilde{P}_k) 
$ of marked independent multi-sets $\widetilde{P}_r = (P_r, \mu_r)$ such that the underlying sequence of multi-sets $(P_1, \dots, P_k)$ constitutes an element of $P_k(\bold m, \mathcal{G})$.
\end{defn}

\begin{rem}
    For each $(P_1, \dots, P_k) \in P_k(\bold m, \mathcal{G})$, there are exactly $\prod_{r=1}^k 2^{|\sigma(P_r)|}$ corresponding ordered $k$-tuples $(\widetilde{P}_1, \dots, \widetilde{P}_k) \in \widetilde{P}_k(\bold m, \mathcal{G})$. This is because each vertex in the underlying set  of each multi-set $P_r$ can independently receive one of two signs $+$ or $-$.
\end{rem}
We have the following explicit expression for $_{\mathbf{m}}\Pi_\mathcal{G}^{2}( q).$
\begin{prop}\label{prop2}
Let \(\mathcal{G}=(I,E)\) be a simple graph with countably many vertices $I$ and let
\(\mathbf m\in \mathbb{Z}_{\ge 0}^{I}\) have non-empty finite support. For every \(q\in\mathbb{N}\), we have
$$
_\mathbf{m}\Pi_{\mathcal G}^2(q)
=
\sum_{k\ge1}
\left(
\sum_{(P_1,\ldots,P_k)\in P_k(\mathbf m,G)}
\prod_{r=1}^k 2^{|\sigma(P_r)|}
\right)
\binom qk.
$$
\end{prop}

\begin{proof}

First, fix a subset of colors $\{c_1 < \dots < c_k\} \subseteq \{1, \dots, q\}$ and consider the double-marked multi-colorings $_{\bold m}\Gamma^{2-\mathrm{mark}}_\mathcal{G}$ that exclusively use these $c_i$ to color the vertices of $\mathcal{G}$. We claim that these specific colorings are in natural bijection with $\widetilde{P}_k(\mathbf{m}, \mathcal{G})$. By the Remark above, the cardinality of the set $\widetilde{P}_k(\mathbf{m}, \mathcal{G})$ is precisely $\sum_{(P_1,\ldots,P_k)\in P_k(\mathbf m,G)}\prod_{r=1}^k 2^{|\sigma(P_r)|}$. Summing over all choices of $k$ colors yields $\binom{q}{k}$ combinations, completing the proof once the bijection is established.

\medskip
Let $_{\bold{m}}\Gamma^{2-\mathrm{mark}}_\mathcal{G}$ be a double-marked multi-coloring of $\mathcal{G}$ using only the given colors. Now we can associate a natural element $(\widetilde{P}_1, \dots, \widetilde{P}_k)\in \widetilde{P}_k(\mathbf{m},\mathcal{G})$ as follows. Denoting by $a_{i,r}$ the multiplicity of color $c_r$ in $M_i$ where $_{\bold{m}}\Gamma^{2-\mathrm{mark}}_\mathcal{G}(i) = (M_i, \epsilon_i)$, we set 
$$P_r=\{i^{a_{i,r}}: i\in I\},\ \ 1\le r\le k.$$
It follows from the definition of $_{\bold{m}}\Gamma^{2-\mathrm{mark}}_\mathcal{G}$ (especially from the constraint $(2)$ on $M_i$'s) that $P_r\in \mathcal{I}({\mathcal{G}})$ is a non-empty independent multi-set and these multi-sets $P_r,\ 1\le r\le k,$ form a partition of the multi-set $\{i^{m_i}: i\in \mathrm{supp}(\mathbf{m})\}$. Thus $(P_1,\dots,P_k)\in {P}_k(\mathbf{m},\mathcal{G})$. Now for each $P_r$ we associate the marking function $\mu_r:\sigma (P_r)\rightarrow \{+,-\}$ as follows 
$$\mu_r(i)=\epsilon_i(c_r)$$
Pairing each multi-set with its marking function yields $\widetilde{P}_r = (P_r, \mu_r)$, which confirms that $(\widetilde{P}_1, \dots, \widetilde{P}_k) \in \widetilde{P}_k(\mathbf{m}, \mathcal{G})$.

Conversely, given any $(\widetilde{P}_1, \dots, \widetilde{P}_k)\in \widetilde{P}_k(\mathbf{m},\mathcal{G})$ we
assign the color $c_r$ to the vertices in $P_r$ counted with multiplicity, for each $1\le r\le k$ and for each pair $(i,c_r)\in I\times [q]$   we assign the sign $\epsilon_i(c_r)=\mu_r(i)$. This gives a double-marked multi-coloring of $\mathcal{G}$ and the  correspondence described  here is bijective.

\end{proof}

\begin{cor}
The function $_{\mathbf{m}}\Pi_\mathcal{G}^{2}( q)$ is a polynomial in \(q\). Moreover, the
coefficient of \(q\) is
$$
_{\mathbf{m}}\Pi_\mathcal{G}^{2}( q)[q]
=
\sum_{k\ge1}\frac{(-1)^{k-1}}{k}\left(
\sum_{(P_1,\ldots,P_k)\in P_k(\mathbf m,G)}
\prod_{r=1}^k 2^{|\sigma(P_r)|}.\right)
$$
\end{cor}

\begin{proof} We get the desired result by noticing that  the coefficient of \(q\) in the  polynomial
    \[
\binom qk
=
\frac{q(q-1)\cdots(q-k+1)}{k!}.
\]
 is given by
\[
\frac{(-1)^{k-1}(k-1)!}{k!}
=
\frac{(-1)^{k-1}}{k}.
\]
\end{proof}

\iffalse
\begin{proof}
Polynomiality follows immediately from the previous proposition, since each \(\binom qk\) is a
polynomial in \(q\). Also
\[
\binom qk
=
\frac{q(q-1)\cdots(q-k+1)}{k!}.
\]
The coefficient of \(q\) in this polynomial is
\[
\frac{(-1)^{k-1}(k-1)!}{k!}
=
\frac{(-1)^{k-1}}{k}.
\]
Therefore
\[
[q]\,m\Pi_G^{2\text{-mark}}(\mathbf m;q)
=
\sum_{k\ge1}
W_k^{(2)}(\mathbf m,G)[q]\binom qk
=
\sum_{k\ge1}\frac{(-1)^{k-1}}{k}W_k^{(2)}(\mathbf m,G).
\]
\end{proof}\fi
The following is the main theorem of this section, which gives a similar expression as in  \cite[Theorem 1]{CDV} for double-marked independence series
\begin{thm}
Let $\mathcal{G}$ be a simple graph with a countable vertex set $I$. For every \(q\in \mathbb{Z}\), we have as formal series
$$ I_{mark}^2(\mathcal{G}, \bold x)^q =
\sum_{\mathbf {m}\in\mathbb{Z}_{\ge0}^{I}} {_\mathbf{m}}\Pi_\mathcal G^{2}(q)\mathbf x^{\mathbf m}.
$$

\end{thm}

\begin{proof}
For any formal power series $f \in \mathcal R$ with a constant term $1$, we have
\[
f^q = \sum_{k \ge 0} \binom{q}{k}(f-1)^k,
\]
where for negative integers we use the standard convention $\binom{-n}{k} = (-1)^k \binom{n+k-1}{k}$ for $n \in \mathbb{N}$.

Set $f=I(\mathcal{G}, \frac{2\mathbf x}{1-\mathbf x})$, then using $\frac{2 x}{1- x}=\sum\limits_{k\ge 1}2x^k$, we get \iffalse and note that 
\begin{align*}
    f-1 &= I(\mathcal{G}, \frac{2\mathbf x}{1-\mathbf x})-1 = \sum_{S\in \mathcal{I(G)}\setminus\{\emptyset\}}\prod_{i\in S}\left(\frac{2x_i}{1-x_i}\right)\\
    &=\sum_{S\in \mathcal{I(G)}\setminus\{\emptyset\}}\prod_{i\in S}\left(2\sum_{l_i\ge 1}x_i^{l_i}\right) = \sum_{S\in \mathcal{I(G)}\setminus\{\emptyset\}} 2^{|S|}\prod_{i\in S}\left(\sum_{l_i\ge 1}x_i^{l_i}\right)\\
    &=\sum_{S\in \mathcal{I}_{mult}(\mathcal{G})\setminus\{\emptyset \}}2^{|\operatorname{supp}(S)|}\left(\prod_{i\in S}x_i\right)
\end{align*}\fi
 $$ f-1 =\sum_{S\in \mathcal{I}_{mult}(\mathcal{G})\setminus\{\emptyset \}}2^{|\sigma(S)|}\left(\prod_{i\in S}x_i\right).$$

Thus, for $k\geq 1$ and $\mathbf{m}\neq 0$, the coefficient $(f-1)^k[x^\mathbf{m}]$ is given by 
        $$\sum_{(S_1,S_2, \ldots, S_k)}\prod_{r=1}^k 2^{|\sigma(S_r)|}$$
        where the sum ranges over all $k$-tuples  $(S_1, S_2, \ldots, S_k)$ satisfying the following conditions:
    \begin{enumerate}
        \item $S_r$ is multi-set and $S_r \in \mathcal{I}_{\mathrm{mult}}(\mathcal{G})\setminus \{\emptyset\}$ for all $1 \leq r \leq k;$
        \item the disjoint union (as multi-sets) of $S_1, \ldots, S_k$ is equal to the multi-set $$\{\underbrace{i, \ldots,i}_{m_i\text{-times}} : i\in\mathrm{supp}(\bold{m})\}.$$
    \end{enumerate}
    It follows that $(S_1, \dots, S_k) \in P_k(\mathbf{m}, \mathcal{G})$, and each element of $P_k(\mathbf{m}, \mathcal{G})$  is obtained in this way. Consequently, the sum ranges over all elements of $P_k(\mathbf{m}, \mathcal{G})$, and we obtain

     $$\sum_{(S_1,\dots,S_k)\in P_k(\mathbf{m},\mathcal{G})}\prod_{r=1}^k 2^{|\sigma(S_r)|}.$$ 
    Therefore, for all $\mathbf{m} \neq 0$, Proposition~\ref{prop2} gives

        $$I\left(\mathcal{G}, \frac{2\mathbf x}{1-\mathbf x}\right)^q [\mathbf{x}^\mathbf{m}]=\sum_{k\geq 1}\binom{q}{k}(f-1)^k[\mathbf{x}^\mathbf{m}] ={_{\bold{m}}\Pi^{2}_\mathcal{G}(q)}.$$
This completes the proof; the coefficients for $\mathbf{m}=0$ clearly coincide.
\end{proof}

We have the following immediate corollary:
\begin{cor}
For every \(q\in\mathbb{Z}\),
\[
H_{\RAAG(\mathcal G)}(\mathbf x)
%I_{mark}^2(\mathcal G, -\mathbf x)^{-1}
=
\sum_{\mathbf m\in\mathbb{Z}_{\ge0}^{I}}
(-1)^{|\mathbf m|}
{_\mathbf{m}}\Pi_\mathcal G^{2}(-1)
\mathbf x^{\mathbf m}.
\] \qed
\end{cor}

\subsection{} In this section, we write the double-marked chromatic polynomials as sum of ordinary chromatic polynomials, this would help us to compute these polynomials explicitly in terms of graphs for large family of graphs. We denote \(\pi_H(q)\) by the ordinary chromatic polynomial of a finite simple
graph \(H\).

\begin{defn} Let $\mathbf m=(m_i:i\in I)\in\mathbb{Z}_{\ge0}^{I}$ be a tuple of non-negative integers with finite support. We define the join of $\mathcal{G}$ with respect to $\bold{m}$, denoted by $\mathcal{G}(\bold{m})$, as follows: replace
vertex \(i\) by a clique of size \(m_i\) (i.e., the completed graph of size $m_i$) and join each vertex of \(i\)-th clique with every vertex of \(j\)-th clique by edges if
 \((i,j)\in E\), otherwise do  not add any edges between those cliques. \end{defn}

 \subsection{}\label{sect5.3}
For a partition \(\lambda_i\vdash m_i\), let \(\ell(\lambda_i)\) denotes its length, and let
\(d_r^{\lambda_i}\) be the number of parts of \(\lambda_i\) equal to \(r\). Define
\[
S(\mathbf m)
=
\left\{
\lambda=(\lambda_i: i\in I): \lambda_i\vdash m_i \text{ for all } i\in\supp(\mathbf m)
\right\}.
\]
For \(\lambda\in S(\mathbf m)\), define
\[
s(\lambda)=\big(\ell(\lambda_i):i\in I\big),
\qquad
L(\lambda)=\sum_{i\in\supp(\mathbf m)}\ell(\lambda_i).
\]

\begin{prop}
Let \(\mathcal G=(I,E)\) be a simple graph and let $\mathbf m=(m_i:i\in I)\in\mathbb{Z}_{\ge0}^{I}$ be a tuple of non-negative integers with finite support. Then
\[
{_\bold m}\Pi_\mathcal G^{2}(q)
=
\sum_{\lambda\in S(\mathbf m)}
2^{L(\lambda)}
\frac{
\pi_{G(s(\lambda))}(q)
}{
\displaystyle
\prod_{i\in\supp(\mathbf m)}
\prod_{r\ge1} d_r^{\lambda_i}!
}.
\]

\end{prop}

\begin{proof}
Let ${_{\bold{m}}\Gamma^{2\text{-mark}}_\mathcal{G}}$ be a double-marked multi-coloring of $\mathcal{G}$ associated to $\bold{m}$ using at most $q$ colors.
By definition, for each vertex $i$,  ${_{\bold{m}}\Gamma^{2\text{-mark}}_\mathcal{G}}(i)$ form a multi-subset of $\{1,\dots,q\}$ of size $m_i$, where each distinct color appearing at ${_{\bold{m}}\Gamma^{2\text{-mark}}_\mathcal{G}}(i)$ is additionally assigned a sign from $\{+,-\}$.  We associate an element $\boldsymbol{\lambda} = (\boldsymbol{\lambda}_i)_{i\in I}\in S(\bold m)$ and a family of ordinary vertex colorings of $\mathcal{G}(\bold s(\boldsymbol{\lambda}))$ as follows. If $m_i=0$, then $\boldsymbol{\lambda}_i=\emptyset$. Otherwise $i\in \mathrm{supp}(\mathbf{m})$ and let the underlying multi-set of colors at $i$ (forgetting the signs) be 
$${_{\bold{m}}\Gamma_{\mathcal{G}}}(i)=\{1^{a_{i,1}},\dots,q^{a_{i,q}}\}.$$

Then $\boldsymbol{\lambda}_i=(\lambda^1_i\ge \cdots \ge \lambda^{\ell(\boldsymbol{\lambda}_i)}_i>0)$ is defined to be the unique partition of $m_i$ associated to the set $\{a_{i,1},\dots,a_{i,q}\}$. Now consider the graph $\mathcal{G}(\bold s(\boldsymbol{\lambda}))$ and fix an ordering on the vertices of each clique (recall that the size of the cliques are given by $\ell(\boldsymbol{\lambda}_i)$). For any choice of distinct colors $c_1,\dots,c_{\ell(\boldsymbol{\lambda}_i)}$ with
\begin{equation}\label{4re}
\lambda^1_i=a_{i,c_1},\dots,\lambda^{\ell(\boldsymbol{\lambda}_i)}_i=a_{i,c_{\ell(\boldsymbol{\lambda}_i)}}
\end{equation}
we color the first vertex of the $i$-th clique of $\mathcal{G}(\bold s(\boldsymbol{\lambda}))$ by $c_1$, the second vertex with $c_2$ and so on. This gives an ordinary vertex coloring. It is easy to see that the number of choices of colors satisfying \eqref{4re} and thus the number of ordinary colorings we get for this fixed underlying multi-set is exactly $\prod_{i\in \mathrm{supp}(\bold m)}\prod_{k=1}^{\infty}(d^{\boldsymbol{\lambda}_i}_{k})!$.

Now conversely let $\boldsymbol{\lambda} = (\boldsymbol{\lambda}_i)_{i\in I}\in S(\bold m)$ and fix an ordering on the vertices of $\mathcal{G}(\bold s(\boldsymbol{\lambda}))$. Let ${\tau}_{\mathcal{G}(\bold s(\boldsymbol{\lambda}))}$ be a vertex coloring of $\mathcal{G}(\bold s(\boldsymbol{\lambda}))$ and suppose that the first vertex of the $i$-th clique receives color $c_1$, the second $c_2$ and so on.
We associate a family of double-marked multi-colorings ${_{\bold{m}}\Gamma^{2\text{-mark}}_\mathcal{G}}$ of $\mathcal{G}$ associated to $\bold{m}$ as follows. First, define the underlying multi-set of colors at $i$ by 
$${_{\bold{m}}\Gamma_{\mathcal{G}}}(i)=\{c_1^{\lambda^1_i},\dots,c_r^{\lambda^r_i}\},\ \text{ where $\boldsymbol{\lambda}_i=(\lambda^1_i\ge \cdots \ge \lambda^r_i>0)$}.$$
To obtain a double-marked multi-coloring, we must assign a sign from $\{+,-\}$ to each distinct color appearing at $i$. Since there are $r=\ell(\boldsymbol{\lambda}_i)$ distinct colors at vertex $i$, there are $2^{\ell(\boldsymbol{\lambda}_i)}$ independent choices of signs at $i$. Over all $i\in I$, this gives $\prod_{i\in I} 2^{\ell(\boldsymbol{\lambda}_i)} = 2^{L(\boldsymbol{\lambda})}$ choices of signs.

Again, if $\lambda_{i}^{r_1}=\lambda_{i}^{r_2}$, then the vertex coloring obtained from ${\tau}_{\mathcal{G}(\bold s(\boldsymbol{\lambda}))}$ by interchanging the colors $c_{r_1}$ and $c_{r_2}$ in the $i$-th clique gives the same double-marked multi-coloring. Hence $\prod_{i\in \mathrm{supp}(\bold m)}\prod_{k=1}^{\infty}(d^{\boldsymbol{\lambda}_i}_{k}!)$ number of usual vertex colorings of $\mathcal{G}(\bold s(\boldsymbol{\lambda}))$ correspond to exactly $2^{L(\boldsymbol{\lambda})}$ double-marked multi-colorings of $\mathcal{G}$ associated to $\bold{m}$. Taking the sum over all such tuples of partitions $\boldsymbol{\lambda}\in S(\bold m)$ gives ${_{\bold{m}}\Pi^{2}_\mathcal{G}}(q)$, and the claim follows.

\end{proof}

We also have similar expression for marked chromatic polynomials in terms of ordinary chromatic polynomials (see \cite[Proposition 3.3]{CDV}).
\begin{prop}
Let \(\mathcal G=(I,E)\) be a simple graph and let $\mathbf m=(m_i:i\in I)\in\mathbb{Z}_{\ge0}^{I}$ be a tuple of non-negative integers with finite support. Then
\[
{_\bold m}\Pi_\mathcal G^{mark}(q)
=
\sum_{\lambda\in S(\mathbf m)}
\frac{
\pi_{G(s(\lambda))}(q)
}{
\displaystyle
\prod_{i\in\supp(\mathbf m)}
\prod_{r\ge1} d_r^{\lambda_i}!
}.
\]

\end{prop}

\section{Chordal graphs and $PEO$-graphs and explicit formulas}\label{sect6}
In this section, we give explicit formulas for the growth series of right angled Artin groups and right angled Coxeter groups associated to chordal graphs or more generally for $PEO$-graphs. 
We begin with recalling the definition of chordal graphs and $PEO$-graphs. 
\begin{defn}
\begin{enumerate}
    \item 
    A \textit{chordal graph} is a finite simple graph in which all cycles of length four or more have an edge that is not
part of the cycle but it connects two vertices of the cycle. Such an edge is called a chord. 
\item A total ordering $\leq $ on $I$ (the vertices of $\mathcal{G}$) is called a \textit{perfect elimination ordering} if for each $r\in I$, the subgraph $\mathcal{G}_r$ induced by the set of vertices 
 $$\{i\in I\setminus \{r\}: \{i,r\}\in E, i<r\}\cup \{r\}$$ 
 is a finite complete subgraph. 
 
 \item A graph $\mathcal{G}$ is called a $PEO$-graph if there exists a perfect elimination ordering on its vertices.

\end{enumerate}
\end{defn}
It is well-known that a finite simple graph is chordal if and only it is a $PEO$-graph. 

\subsection{}\label{sect6.1} We have explcit formulas for marked chromatic polynomials, it was  first proved in \cite[Theorem 2]{CDV}.
\begin{thm}\label{peomarked}
    Let $\mathcal{G}$ be a PEO-graph with a countable vertex set $I$. We have
     \begin{equation}
    {_{\bold{m}}\Pi^{\mathrm{mark}}_\mathcal{G}}(q)=\sum\limits_{\boldsymbol{\lambda}\in S(\bold m)}\prod _{j\in \mathrm{supp}(\bold{m})} \binom{q-b^{\boldsymbol{\lambda}}_{j}}{\ell(\boldsymbol{\lambda}_{j})} \frac{\ell(\boldsymbol{\lambda}_{j})!}{\prod_{k=1}^{\infty}(d^{\boldsymbol{\lambda}_{j}}_{k}!)},\ \ b^{\boldsymbol{\lambda}}_{j}=\sum\limits_{\substack{i\in \mathcal{G}_j\backslash\{j\}}}\ell(\boldsymbol{\lambda}_{i})
        \end{equation}
        \end{thm}
The above theorem has the following immediate corollary (see \cite[Corollary 6.3]{CDV}. We fix the perfect elimination ordering on $I$ and  set $\mathcal{G}_r$ by the subgraph of $\mathcal{G}$ induced by the set of vertices 
 $\{i\in I\setminus \{r\}: \{i,r\}\in E, i<r\}\cup \{r\}$ in what follows next.
        \begin{cor}\label{racg}
 Let $\mathcal{G}$ be a PEO-graph with a countable vertex set $I$. Then we have 
$$H_{RACG(\mathcal{G})}(\bold x) =\sum\limits_{\boldsymbol{m}\in \mathbb{Z}^I_+}\left(\sum\limits_{\boldsymbol{\lambda}\in S(\bold m)}(-1)^{|\mathbf{m}|+\sum_{j\in \mathrm{supp}(\bold{m})}\ell(\boldsymbol{\lambda}_{j})}\prod _{j\in \mathrm{supp}(\bold{m})} \binom{\sum\limits_{i\in \mathcal{G}_j}\ell(\boldsymbol{\lambda}_{i})}{\ell(\boldsymbol{\lambda}_{j})} \frac{\ell(\boldsymbol{\lambda}_{j})!}{\prod_{k=1}^{\infty}(d^{\boldsymbol{\lambda}_{j}}_{k}!)}\right) x^\mathbf{m}$$ \qed
 
\end{cor}

\subsection{}
In the following theorem, we  determine the double-marked chromatic polynomials for $PEO$ graphs (which include finite chordal graphs).

\begin{thm}\label{peomarked}
    Let $\mathcal{G}$ be a PEO-graph with a countable vertex set $I$. We have
     \begin{equation}
  {_\bold m}\Pi_\mathcal G^{2}(q)=\sum\limits_{\boldsymbol{\lambda}\in S(\bold m)}2^{L(\lambda)}\prod _{j\in \mathrm{supp}(\bold{m})} \binom{q-b^{\boldsymbol{\lambda}}_{j}}{\ell(\boldsymbol{\lambda}_{j})} \frac{\ell(\boldsymbol{\lambda}_{j})!}{\prod_{k=1}^{\infty}(d^{\boldsymbol{\lambda}_{j}}_{k}!)},\ \ b^{\boldsymbol{\lambda}}_{j}=\sum\limits_{\substack{i\in \mathcal{G}_j\backslash\{j\}}}\ell(\boldsymbol{\lambda}_{i})
        \end{equation}
        \end{thm}

        \begin{proof}

     Let $\mathrm{supp}(\bold{m})=\{i_1<\dots<i_N\}$ with respect to the chosen PEO-ordering. It is enough to consider the induced subgraph on  $\mathrm{supp}(\bold{m})$, which we relabel as $\{1,\dots,N\}$. For simplicity, we call this subgraph in the rest of the proof also by $\mathcal{G}$. Inheriting the induced ordering of the graph, this is again a PEO-graph. We will consider colorings corresponding to $\boldsymbol{\lambda}\in S(\bold m)$, that has the form  
     \begin{equation}\label{form}{_{\bold{m}}\Gamma^{\mathrm{mark}}_\mathcal{G}}(i)=\{c_1^{\lambda^1_i},\dots,c_{r_i}^{\lambda^{r_i}_i}\},\ \text{ $\boldsymbol{\lambda}_i=(\lambda^1_i\ge \cdots \ge \lambda^{r_i}_i>0)$}\end{equation}
     for some $c_1,\dots,c_{r_i}\in\{1,\dots,q\}.$
     
 To color the first vertex, we have 
     $$\binom{q}{\ell(\boldsymbol{\lambda}_{1})}\frac{\ell(\boldsymbol{\lambda}_{1})!}{\prod_{k=1}^{\infty}(d^{\boldsymbol{\lambda}_{1}}_k!)}$$ possibilities. To obtain double-marked multi-coloring we have to assign sign from $\{+,-\}$ to each distinct color appearing at $1$. Since there are $\ell(\lambda_1)$ distinct colors at the vertex $1$, there are $2^{\ell(\lambda_1)}$ independent choices of signs at the vertex $1.$ %  Let $j_1<\dots<j_p$ are the vertices in $I_N$ such that $j_m\in\Psi_0$ for all $m\in\{1,\dots,p\}$. First we will consider the coloring corresponding to $\bar{\lambda}=(\lambda_i)\in S(\bold m)$. When coloring the vertices of $\mathcal{G}(I_N)$, if the first vertex corresponds to odd isotropic root ($1\in\Psi_0)$ then we can color this vertex with $l(\lambda_{j_1})$ distinct colors where $l(\lambda_{j_1})$ is the length of a partition $\lambda_{j_1}=(\lambda^1_{j_1}\geq \dots\geq\lambda^{k_{j_1}}_{j_1})$ of $m_{j_1}=m_1$. We have to arrange $l(\lambda_{j_1})$ distinct colors such that the first placed color repeats $\lambda^1_{j_1}$ times, second placed color   repeats $\lambda^2_{j_1}$ times, and so on $k_{j_1}^{th}$ placed color repeats $\lambda^{k_{j_1}}_{j_1}$ times, but there can be some $k=\lambda^s_{j_1}$  repeats $q^{\lambda_{j_1}}_k$ number of times in $\lambda_{j_1}$, so we can color the first vertex with $\frac{\Perm{q}{l(\lambda_{j_1})}}{\prod_{t=1}^{\infty}q^{\lambda_{j_1}}_t!}$ ways, where $q^{\lambda_{j_1}}_t$ be the number of times $t$ appears in the partition  $\lambda_{j_1}$. If the first vertex $1\notin \Psi_0$, then we can choose  $m_1$ colors from the given $q$ colors and can color it. 
   Now suppose that we have colored the first $n-1$ vertices, $2\leq n<N$. Note that the subgraph $\mathcal{G}_n$ is complete. So to color vertex $n$, we cannot use the colors that were used to color the vertices in $\mathcal{G}_n \setminus \{n\}$. Thus, we must choose $\ell(\boldsymbol{\lambda}_n)$ colors out of the remaining $q - b^{\boldsymbol{\lambda}}_{n}$, and we get%So now to color vertex $n$ we can not use the colors that are used to color the vertices in $\mathcal{G}_n\backslash\{n\}$. Therefore we have to choose $\ell(\boldsymbol{\lambda}_n)$ colors out of the remaining $q-b^{\boldsymbol{\lambda}}_{n}$ and we get 
   $$\binom{q-b^{\boldsymbol{\lambda}}_{n}}{\ell(\boldsymbol{\lambda}_n)}\frac{\ell(\boldsymbol{\lambda}_n)!}{\prod_{k=1}^{\infty}(d^{\boldsymbol{\lambda}_{n}}_k!)}$$
   choices to color vertex $n$. Now to obtain a double-marked multi-coloring, we assign a sign from $\{+, -\}$ to each distinct color appearing at vertex $n$. Since there are $\ell(\boldsymbol{\lambda}_n)$ distinct colors at vertex $n$, there are $2^{\ell(\boldsymbol{\lambda}_n)}$ independent choices of signs at vertex $n$. Therefore, the total number of double-marked multi-colorings corresponding to $\boldsymbol{\lambda}$ is given by
$$2^{L(\boldsymbol{\lambda})} \prod_{j \in \mathrm{supp}(\mathbf{m})} \binom{q - b^{\boldsymbol{\lambda}}_{j}}{\ell(\boldsymbol{\lambda}_{j})} \frac{\ell(\boldsymbol{\lambda}_{j})!}{\prod_{k=1}^{\infty}(d^{\boldsymbol{\lambda}_{j}}_{k}!)}.$$
Taking the sum over all such partitions $\boldsymbol{\lambda} \in S(\mathbf{m})$ gives the result.

        \end{proof}

In the following corollary we obtain the Growth series of RAAG for PEO-graphs.
\begin{cor}\label{peohilbert}
 Let $\mathcal{G}$ be a PEO-graph with a countable vertex set $I$. Then we have 
 $$H_{\RAAG(\mathcal G)}(\mathbf x)=\sum\limits_{\boldsymbol{m}\in \mathbb{Z}^I_+}\left(\sum\limits_{\boldsymbol{\lambda}\in S(\bold m)}2^{L(\boldsymbol{\lambda})}(-1)^{|\mathbf{m}|+\sum_{j\in \mathrm{supp}(\bold{m})}\ell(\boldsymbol{\lambda}_{j})}\prod _{j\in \mathrm{supp}(\bold{m})} \binom{\sum\limits_{i\in \mathcal{G}_j}\ell(\boldsymbol{\lambda}_{i})}{\ell(\boldsymbol{\lambda}_{j})} \frac{\ell(\boldsymbol{\lambda}_{j})!}{\prod_{k=1}^{\infty}(d^{\boldsymbol{\lambda}_{j}}_{k}!)}\right) \bold x^\mathbf{m}$$ \qed
 
\end{cor}
For large families of graphs, such as PEO-graphs, the formulas in Corollary~\ref{racg} and Corollary~\ref{peohilbert} are closed and computable, allowing explicit computation of the growth series of RACGs and RAAGs in these cases.

\section{Closed-Form Growth Series for special classes of $PEO$ Graphs}\label{sect7}
In this section, we apply Corollary~\ref{peohilbert} to derive explicit multivariate growth series for standard families of $PEO$ graphs.

\subsection{Complete Graphs}

Let $K_n$ be the complete graph on $n$ vertices equipped with the ordering $1 < 2 < \dots < n$. Since $N(j) \cap \{1, \dots, j-1\} = \{1, 2, \dots, j-1\}$ for every $j$, we have $\mathcal{G}_j = \{1, 2, \dots, j\}$. Substituting this into Corollary~\ref{peohilbert}, the growth series for $\mathrm{RAAG}(K_n) $ is given by
\begin{equation*}
H_{\mathrm{RAAG}(K_n)}(\mathbf{x})[\mathbf{x}^{\mathbf{m}}] =  \sum_{\boldsymbol{\lambda} \in S(\mathbf{m})} 2^{L(\boldsymbol{\lambda})} (-1)^{|\mathbf{m}| + L(\boldsymbol{\lambda})} \left(\prod_{j \in \mathrm{supp}(\mathbf{m})} \binom{\sum_{i=1}^{j} \ell(\boldsymbol{\lambda}_{i})}{\ell(\boldsymbol{\lambda}_{j})} \frac{\ell(\boldsymbol{\lambda}_{j})!}{\prod_{k=1}^{\infty} \left(d^{\boldsymbol{\lambda}_{j}}_{k}!\right)}\right).
\end{equation*}

\subsection{Trees}

Every finite connected tree $T = (I, E)$ is chordal and therefore admits a perfect elimination ordering. Formally, a $PEO$ $1 < 2 < \dots < n$ on $T$ can be constructed by reversing a simplicial elimination sequence, where at each step a vertex of degree $1$ in the remaining induced subgraph is removed. 

Consequently, for any $j \in \{2, \dots, n\}$, the set of preceding neighbors $N(j) \cap \{1, \dots, j-1\}$ contains exactly one vertex, which we denote by $\pi(j) \in \{1, \dots, j-1\}$. Therefore, the  set $\mathcal{G}_j = \{\pi(j), j\}$  has cardinality $2$ for all $j > 1$. Setting $\ell(\boldsymbol{\lambda}_{\pi(1)}) = 0$ by convention, Corollary~\ref{peohilbert} reduces to
\begin{equation*}
H_{\mathrm{RAAG}(T)}(\mathbf{x})[\mathbf{x}^{\mathbf{m}}] = \sum_{\boldsymbol{\lambda} \in S(\mathbf{m})} 2^{L(\boldsymbol{\lambda})} (-1)^{|\mathbf{m}| + L(\boldsymbol{\lambda})} \left(\prod_{j \in \mathrm{supp}(\mathbf{m})} \binom{\ell(\boldsymbol{\lambda}_{\pi(j)}) + \ell(\boldsymbol{\lambda}_{j})}{\ell(\boldsymbol{\lambda}_{j})} \frac{\ell(\boldsymbol{\lambda}_{j})!}{\prod_{k=1}^{\infty} \left(d^{\boldsymbol{\lambda}_{j}}_{k}!\right)} \right).
\end{equation*}

\subsection{Paths and Star Graphs}

Since path graphs and star graphs are specialized trees, their growth series follow directly from the tree formula.

\begin{example}[Path Graphs]
Let $P_n$ be a path graph on $n$ vertices with the standard elimination ordering $1 < 2 < \dots < n$. For each $j \ge 2$, the unique preceding neighbor is $\pi(j) = j-1$. The growth series is given by
\begin{equation*}
H_{\mathrm{RAAG}(P_n)}(\mathbf{x})[\mathbf{x}^{\mathbf{m}}] = \sum_{\boldsymbol{\lambda} \in S(\mathbf{m})} 2^{L(\boldsymbol{\lambda})} (-1)^{|\mathbf{m}| + L(\boldsymbol{\lambda})}  \left(\prod_{j \in \mathrm{supp}(\mathbf{m})} \binom{\ell(\boldsymbol{\lambda}_{j-1}) + \ell(\boldsymbol{\lambda}_{j})}{\ell(\boldsymbol{\lambda}_{j})} \frac{\ell(\boldsymbol{\lambda}_{j})!}{\prod_{k=1}^{\infty} \left(d^{\boldsymbol{\lambda}_{j}}_{k}!\right)} \right),
\end{equation*}
with $\ell(\boldsymbol{\lambda}_{0})= 0$ by convention.
\end{example}

\begin{example}[Star Graphs]
Let $K_{1, n-1}$ be a star graph. Assign the central vertex (of degree $n-1$) to position $1$, and the remaining simplicial vertices (leaves of degree $1$) to positions $2, \dots, n$. Under this ordering, $\pi(j) = 1$ for all $j \ge 2$, yielding
\begin{equation*}
H_{\mathrm{RAAG}(K_{1, n-1})}(\mathbf{x})[\mathbf{x}^{\mathbf{m}}] =  \sum_{\boldsymbol{\lambda} \in S(\mathbf{m})} 2^{L(\boldsymbol{\lambda})} (-1)^{|\mathbf{m}| + L(\boldsymbol{\lambda})} \left( \frac{\ell(\boldsymbol{\lambda}_{1})!}{\prod_{k=1}^{\infty} \left(d^{\boldsymbol{\lambda}_{1}}_{k}!\right)}\prod_{j=2}^n\binom{\ell(\boldsymbol{\lambda}_{1}) + \ell(\boldsymbol{\lambda}_{j})}{\ell(\boldsymbol{\lambda}_{j})} \frac{\ell(\boldsymbol{\lambda}_{j})!}{\prod_{k=1}^{\infty} \left(d^{\boldsymbol{\lambda}_{j}}_{k}!\right)} \right) \mathbf{x}^{\mathbf{m}}.
\end{equation*}
\end{example}

\subsection{Edgeless Graphs}

Let $\overline{K_n}$ be the edgeless graph on $n$ vertices. Any ordering $$1 < 2 < \dots < n$$ is a PEO. Since $N(j) = \emptyset$ for all $j$, we have $\mathcal{G}_j = \{j\}$. The binomial coefficient reduces to $\binom{\ell(\boldsymbol{\lambda}_j)}{\ell(\boldsymbol{\lambda}_j)} = 1$, and the growth series for $\mathrm{RAAG}(\overline{K_n}) \cong \mathbb{Z}^n$ is given by 
\begin{equation*}
H_{\mathrm{RAAG}(\overline{K_n})}(\mathbf{x})[\mathbf{x}^{\mathbf{m}}] = \sum_{\boldsymbol{\lambda} \in S(\mathbf{m})} 2^{L(\boldsymbol{\lambda})} (-1)^{|\mathbf{m}| + L(\boldsymbol{\lambda})}  \left(\prod_{j \in \mathrm{supp}(\mathbf{m})} \frac{\ell(\boldsymbol{\lambda}_{j})!}{\prod_{k=1}^{\infty} \left(d^{\boldsymbol{\lambda}_{j}}_{k}!\right)} \right) \mathbf{x}^{\mathbf{m}}.
\end{equation*}

\end{document}